\documentclass{amsart}
\usepackage{amssymb,amsmath,amscd,amsthm}
\usepackage{amsfonts,epsfig,latexsym,graphicx,amssymb}

\date{\today}

\newtheorem{theorem}{Theorem}[section]
\newtheorem{lemma}[theorem]{Lemma}

\newtheorem{defi}[theorem]{Definition}

\theoremstyle{definition}
\newtheorem{remark}[theorem]{Remark}

\def\be{\begin{equation}}
\def\ee{\end{equation}}

\renewcommand{\includegraphics}[2][]{}

\begin{document}

\title[Path connectedness and entropy density]{Path connectedness and entropy density of the space of hyperbolic ergodic measures}

\author[A.\ Gorodetski]{Anton Gorodetski}

\address{Department of Mathematics, University of California, Irvine, CA~92697, USA}

\email{asgor@math.uci.edu}

\thanks{A.\ G.\ was supported in part by NSF grant DMS--1301515}   

\author[Ya.\ Pesin]{Yakov Pesin}

\address{Department of Mathematics, Penn State University, University Park, PA 16802, USA}

\email{pesin@math.psu.edu}

\thanks{Ya.\ P.\ was supported in part by NSF grant DMS--1400027.}


\begin{abstract}
We show that the space of hyperbolic ergodic measures of a given index supported on an isolated homoclinic class is path connected and entropy dense provided that any two hyperbolic periodic points in this class are homoclinically related. As a corollary we obtain that the closure of this space is also path connected.
\end{abstract}

\maketitle

\section{Introduction}\label{s.intro}

In this paper we consider homoclinic classes of periodic points for $C^{1+\alpha}$ diffeomorphisms of compact manifolds and we discuss two properties of the space of invariant measures supported on them and equipped with the  weak$^*$-topology -- connectedness and entropy density of the subspace of hyperbolic ergodic measures. The study of connectedness of the latter space was initiated by Sigmund in a short article \cite{S77}. He established path connectedness of this space in the case of transitive topological Markov shifts and as a corollary, of transitive Axiom $A$ diffeomorphisms. Sigmund's idea was to show first that any two periodic measures (i.e., invariant atomic measures on periodic points) can be connected by a continuous path of ergodic measures and second that if one of the two periodic measures lies in a small neighborhood of another one, then the whole path can be chosen to lie in this neighborhood. In order to carry out the first step Sigmund shows that any periodic measure can be approximated by a Markov measure and that any two Markov measures can be connected by a path of Markov measures. We use Sigmund's idea in our proof of Theorem \ref{t.main}.

A different approach to Sigmund's theorem is to show that ergodic measures on a transitive topological Markov shift are dense in the space of all invariant measures. Since the latter space is a Choquet simplex and ergodic measures are its extremal points, it means that this space is the Poulsen simplex (which is unique up to an affine homeomorphism). The desired result now follows from a complete description of the Poulsen simplex given in \cite{LOS} (see also \cite{GK}). 

Since Sigmund's work the interest to the study of connectedness of the space of hyperbolic ergodic measures has somehow been lost,\footnote{At the time of writing this paper there is no single reference to the paper by Sigmund \cite{S77} in {\it MathSciNet}.} and only recently it has regained attention.\footnote{Soon after this paper was completed, several new works related to the subject appeared, see \cite{BZ, DGR2}.} In \cite{GT} Gogolev and Tahzibi, motivated by their study of existence of non-hyperbolic invariant measures, raised a question of whether the space of ergodic measures invariant under some partially hyperbolic systems is path connected. This includes, in particular, the famous example by Shub and Wilkinson \cite{SW}. Some results on connectedness and other topological properties of the space of invariant measures were obtained in \cite{BG, GK}.

All known proofs of connectedness of the space of invariant measures are based on approximating invariant measures by either measures supported on periodic orbits or Markov ergodic measures supported on invariant horseshoes. It is therefore natural to ask whether such approximations can be arranged to also ensure convergence of entropies. If this is possible, the space of approximants is called {\it entropy dense}. Some results in this direction were obtained in \cite{FO}. We stress that approximating {\it hyperbolic ergodic} measures with positive entropy by ``nice'' measures supported on invariant horseshoes so that the convergence of entropies is also guaranteed, was first done by Katok in \cite{K} (see also \cite{BP, KH}). We use this result in the proof of our Theorem \ref{t.entropydense} where we approximate also some hyperbolic ergodic measures with zero entropy as well as non-ergodic measures.

We shall now state our results. Consider a $C^{1+\alpha}$-diffeomorphism $f:M\to M$ of a compact smooth manifold $M$. Let $p\in M$ be a hyperbolic periodic point. By the index $s(p)$ of $p$ we mean the dimension of the invariant stable manifold of $p$.

We say that a hyperbolic periodic point $q\in M$ is \emph{homoclinically related} to $p$ and write $q\sim p$ if the stable manifold of the orbit of $q$ intersects  transversely the unstable manifold of the orbit of $p$ and vice versa.\footnote{The stable (respectively, unstable) manifold of the orbit of a periodic point is the union of stable (respectively, unstable) manifolds through every point on the orbit. If $q\sim p$, then $s(q)=s(p)$.} Notice that this is an equivalence relation. We denote by $\mathcal{H}(p)$ the \emph{homoclinic class} associated with the point $p$, that is the closure of the set of hyperbolic periodic points homoclinically related to $p$. Note that $\mathcal{H}(p)$ is $f$-invariant. Homoclinic classes were introduced by Newhouse in \cite{New}.

A basic hyperbolic set of an Axiom $A$ diffeomorphism gives the simplest example of a homoclinic class, but in general the set $\mathcal{H}(p)$ can have a much more complicated structure and dynamical properties. In particular, it can contain non-hyperbolic periodic points, and it can support non-hyperbolic (periodic or not) measures in a robust way, see \cite{BBD, BDGor, DG,  KN} for a more detailed discussion.\footnote{It is conjectured  that existence of non-hyperbolic ergodic measures is a characteristic property of non-hyperbolic homoclinic classes, see \cite{BBD, DG}.} It can also contain in a robust way hyperbolic periodic orbits whose index is different than the index of $p$, see \cite{BD, GI, Gor1}. Moreover -- and this is of importance for us in this paper -- there may exist hyperbolic periodic points in $\mathcal{H}(p)$ of the same index as $p$ that are {\it not} homoclinically related to $p$, see \cite{DG, DHRS}. Besides, it can happen that periodic orbits outside the homoclinic class $\mathcal{H}(p)$ accumulate to  $\mathcal{H}(p)$; for example, this is part of the Newhouse phenomena, and also occurs in the family of standard maps, see \cite{Du, Gor}. We wish to avoid both of these complications, and we therefore, impose the following crucial requirements on the homoclinic class $\mathcal{H}(p)$: 
\begin{enumerate}
\item[(H1)] \emph{For any hyperbolic periodic point $q\in\mathcal{H}(p)$ with $s(q)=s(p)$ we have $q\sim p$.}
\item[(H2)] \emph{The homoclinic class $\mathcal{H}(p)$ is isolated, i.e., there is an open neighborhood $U(\mathcal{H}(p))$ of $\mathcal{H}(p)$ such that  
$\mathcal{H}(p)=\bigcap_{n\in \mathbb{Z}}f^n(U(\mathcal{H}(p)))$.}
\end{enumerate}
We stress that these requirements do hold in many interesting cases, see examples in Section \ref{s.examples}. In particular, Condition (H2) holds if the map $f$ has only one homoclinic class. This is the case in Examples 1 and 2  in Section \ref{s.examples}. We also note a result in \cite{BG} that is somewhat related to Condition (H2): if the map $f$ admits a dominated splitting of index $s$, then a linear combination of hyperbolic ergodic measures of index $s$ can be approximated by a sequence of hyperbolic ergodic measures of index $s$ if and only if their homoclinic classes coincide.

The space of all invariant ergodic measures supported on $\mathcal{H}(p)$ can be extremely rich and contain hyperbolic measures with different number of positive Lyapunov exponents as well as non-hyperbolic measures. We denote by $\mathcal{M}_p$ the space of all hyperbolic invariant measures supported on $\mathcal{H}(p)$ for which the number of negative Lyapunov exponents at almost every point is exactly $s(p)$. We say that $\mu$ has index $s(p)$. Further, we denote by $\mathcal{M}_p^e$ the space of all hyperbolic \emph{ergodic} measures in $\mathcal{M}_p$. We assume that the space $\mathcal{M}_p$ is equipped with the weak$^*$-topology.

\begin{theorem}\label{t.main}
Under Conditions (H1) and (H2) the space $\mathcal{M}_p^e$ is path connected.
\end{theorem}

Notice that without Conditions (H1) and (H2) the conclusion of Theorem \ref{t.main} may fail, see Subsection \ref{ss.2}.

It follows immediately from Theorem \ref{t.main} that the closure of
$\mathcal{M}_p^e$ is connected. In fact, a stronger statement holds.

\begin{theorem}\label{t.2}
Under Conditions (H1) and (H2) the closure of the space $\mathcal{M}_p^e$ is 
path connected.
\end{theorem}

\begin{remark}
It is interesting to notice that the closure of $\mathcal{M}_p^e$ is {\it not} a Choquet simplex (and hence, not a Poulsen simplex), see Proposition 2.7 in \cite{BZ}.
\end{remark}

We shall now discuss the entropy density of the space $\mathcal{M}_p^e$.

\begin{defi}
A subset $S\subseteq\mathcal{M}_p$ is {\it entropy dense} in $\mathcal{M}_p$ if for any $\mu\in\mathcal{M}_p$ there exists a sequence of measures
$\{\xi_n\}_{n\in\mathbb{N}}\subset S$ such that $\xi_n\to\mu$ and
$h_{\xi_n}\to h_\mu$ as $n\to \infty$.
\end{defi}

\begin{theorem}\label{t.entropydense}
Under Conditions (H1) and (H2) the space $\mathcal{M}_p^e$ is entropy dense in $\mathcal{M}_p$.
\end{theorem}

\section{Examples}\label{s.examples}

In this section we present some examples that illustrate importance of Conditions (H1) and (H2).

\subsection{Non-hyperbolic homoclinic classes satisfying Conditions (H1) and (H2).} We describe a class of diffeomorphisms with a partially hyperbolic attractor which is the homoclinic class of any of its periodic points and which satisfies Conditions (H1) and (H2). We follow \cite{BDPP}. Let $f$ be a $C^{1+\alpha}$ diffeomorphism of a compact smooth manifold $M$ and $\Lambda$ a topological attractor for $f$. This means that there is an open set $U\subset M$ such that $\overline{f(U)}\subset U$ and $\Lambda=\bigcap_{n\ge 0}f^n(U)$. We assume that $\Lambda$ is a partially hyperbolic set for $f$, that is for every $x\in\Lambda$ there is an invariant splitting of the tangent space $T_xM=E^s(x)\oplus E^c(x)\oplus E^u(x)$ into stable $E^s(x)$, central
$E^c(x)$ and unstable $E^u(x)$ subspaces such that with respect to some Riemannian metric on $M$ we have that for some constants
$$
0<\lambda_1<\lambda_2<\lambda_3<\lambda_4, \quad \lambda_1<1, \quad \lambda_4>1
$$
the following holds:
\begin{enumerate}
\item $\|dfv\|<\lambda_1\|v\|$ for every $v\in E^s(x)$,
\item $\lambda_2\|v\|<\|dfv\|<\lambda_3\|v\|$ for every $v\in E^c(x)$,
\item $\|dfv\|>\lambda_4\|v\|$ for every $v\in E^u(x)$.
\end{enumerate}
If $\Lambda$ is a partially hyperbolic attractor for $f$, then for every
$x\in\Lambda$ we denote by $V^u(x)$ and $W^u(x)$ the local and respectively global unstable leaves through $x$. It is known that for every $x\in\Lambda$ and $y\in W^u(x)$ one has $T_yW^u(x)=E^u(y)$, $f(W^u(x))=W^u(f(x))$ and $W^u(x)\subset\Lambda$. Moreover, the collection of all global unstable leaves $W^u(x)$ forms a continuous lamination of $\Lambda$ with smooth leaves, and if $\Lambda=M$, then it is a continuous foliation of $M$ with smooth leaves.

An invariant measure $\mu$ on $\Lambda$ is called a $u$-measure if the conditional measures it generates on local unstable leaves $V^u(x)$ are equivalent to the leaf volume on $V^u(x)$ induced by the Riemannian metric. It is shown in \cite{PS} that any partially hyperbolic attractor admits a
$u$-measure: any limit measure for the sequence of measures
$$
\mu_n=\frac1n\sum_{k=0}^{n-1}f_*^k m
$$
is a $u$-measure on $\Lambda$. Here $m$ is the Riemannian volume in a sufficiently small neighborhood of the attractor (see \cite{PS} for more details and other ways for constructing $u$-measures).

In general a $u$-measure may have some or all Lyapunov exponents along the central direction to be zero.\footnote{Clearly, the Lyapunov exponents in the stable direction are negative while the Lyapunov exponents in the unstable direction are positive.} Therefore, following \cite{BDPP} we say that a
$u$-measure $\mu$ has negative central exponents on an invariant subset
$A\subset\Lambda$ of positive measure if for every $x\in A$ and
$v\in T_xE^c(x)$ the Lyapunov exponent $\chi(x,v)<0$.

We consider the following requirement on the map $f|\Lambda$:
\begin{enumerate}
\item[(D)] for every $x\in\Lambda$ the positive semi-trajectory of the global unstable leaf $W^u(x)$ is dense in $\Lambda$, that is
$$
\overline{\bigcup_{n\ge 0}\,f^n(W^u(x))}=\overline{\bigcup_{n\ge 0}\,W^u(f^n(x))}=\Lambda.
$$
\end{enumerate}
Condition (D) clearly holds if the unstable lamination is minimal, i.e., if every leaf of the lamination is dense in $\Lambda$. It is shown in \cite{BDPP} that if
$\mu$ is a $u$-measure on $\Lambda$ with negative central exponents on an invariant subset of positive measure and if $f$ satisfies Condition (D), then 1)
$\mu$ has negative central exponents at almost every point $x\in\Lambda$; 2) $\mu$ is the unique $u$-measure for $f$ supported on the whole $\Lambda$; and 3) the basin of attraction for $\mu$ coincides with the open set $U$.

It is easy to see that in this case:
\begin{enumerate}
\item hyperbolic periodic points whose index is equal to the dimension of the stable leaves\footnote{This dimension is $\dim E^s+\dim E^c$.} are dense in the attractor $\Lambda$; the homoclinic class of each of these periodic points coincides with  $\Lambda$;
\item the homoclinic class satisfies Conditions (H1) and (H2), and hence, Theorems \ref{t.main}, \ref{t.2}, and \ref{t.entropydense} are applicable.
\end{enumerate}

Let $f_0$ be a partially hyperbolic diffeomorphism which is either 1) a skew product with the map in the base being a topologically transitive Anosov diffeomorphism or 2) the time-$1$ map of an Anosov flow. If $f$ is a small perturbation of $f_0$ then $f$ is partially hyperbolic and by \cite{HPS}, the central distribution of $f$ is integrable. Furthermore, the central leaves are compact in the first case and there are compact leaves in the second case. It is shown in \cite{BDPP} that $f$ has minimal unstable foliation provided there exists a compact periodic central leaf $\mathcal{C}$ (i.e.,
$f^\ell(\mathcal{C})=\mathcal{C}$ for some $\ell\ge 1$) for which the restriction $f^\ell|\mathcal{C}$ is a minimal transformation.

Furthermore, it follows from the results in \cite{BB} that starting from a volume preserving partially hyperbolic diffeomorphism $f_0$ with one-dimensional central subspace, it is possible to construct a $C^2$ volume preserving diffeomorphism $f$ which is arbitrarily  $C^1$-close to $f_0$ and has negative central exponents on a set of positive volume. Moreover, if $\mathcal{C}$ is a compact periodic central leaf, then $f$ can be arranged to coincide with $f_0$ in a small neighborhood of the trajectory of $\mathcal{C}$.

We now consider the two particular examples.

\vspace{5pt}
{\bf Example 1.} Consider the time-$1$ map $f_0$ of the geodesic flow on a compact surface of negative curvature. Clearly, $f_0$ is partially hyperbolic and has a dense set of compact periodic central leaves. It follows from what was said above that there is a volume preserving perturbation $f$ of $f_0$ such that
\begin{enumerate}
\item $f$ is of class $C^2$ and is arbitrary close to $f_0$ in the $C^1$-topology;
\item $f$ is a partially hyperbolic diffeomorphism with one-dimensional central subspace;
\item there exists a central leaf $\mathcal{C}$ such that the restriction
$f^\ell|\mathcal{C}$ is a minimal transformation (here $\ell$ is the period of the leaf);
\item $f$ has negative central exponents on a set of positive volume;
\item the unstable foliation for $f$ is minimal and hence, satisfies Condition (D).
\end{enumerate}
We conclude that in this example the whole manifold is the homoclinic class of every hyperbolic periodic point of index two and that this class satisfies Conditions (H1) and (H2).

\vspace{5pt}
{\bf Example 2.} Consider the map $f_0=A\times R$ of the $3$-torus
$T^3=T^2\times T^1$ where $A$ is a linear Anosov automorphism of the $2$-torus $T^2$ and $R$ is an irrational rotation of the circle $T^1$. It follows from what was said above that there is a volume preserving perturbation $f$ of $f_0$ such that the properties (1) -- (5) in the previous example hold, and hence the unique homoclinic class satisfies Conditions (H1) and (H2).

\begin{remark}
It was shown in \cite{BDU} that the set of partially hyperbolic diffeomorphisms with one dimensional central direction contains a $C^1$ open and dense subset of diffeomorphisms with minimal unstable foliation. However, in our examples  we use preservation of volume to ensure negative central Lyapunov exponents on a set of positive volume, so we cannot immediately apply the result in \cite{BDU} to obtain an open set of systems for which Conditions (H1) and (H2) hold, compare with Problem 7.25 from \cite{BDV}.
\end{remark}

\begin{remark}
In both Examples 1 and 2 the map possesses a non-hyperbolic ergodic invariant measure (e.g. supported on the compact periodic leaf). We believe that in these examples presence of non-hyperbolic ergodic invariant measures is persistent under small perturbations.  Indeed, since the central subspace is one dimensional, the central Lyapunov exponent with respect to a given ergodic measure is an integral of a continuous function (i.e., log of the expansion rate along the central subspace) over this measure, existence of periodic points of different indices combined with (presumable) connectedness of the space of ergodic measures should imply existence of a non-hyperbolic invariant ergodic measure. See \cite{BBD, BDGor, DGor, GIKN} for the related results and discussion.
\end{remark}

\subsection{Homoclinic classes that do not satisfy Conditions (H1) and (H2).}\label{ss.2}

There is an example of an invariant set for a partially hyperbolic map with one dimensional central subspace which is a homoclinic class containing two non-homoclinically related hyperbolic periodic orbits of the same index, hence, not satisfying Condition (H1), see \cite{DG,  DGR, DHRS}. Moreover, the space of hyperbolic ergodic measures supported on this homoclinic class is not connected due to the fact that the set of all central Lyapunov exponents is split into two disjoint closed intervals, see Remark 5.2 in \cite{DGR}.

As we already mentioned in Introduction, Condition (H2) does not always hold even for surface diffeomorphisms, see for example, \cite{Du, Gor}. This condition ensures that the hyperbolic horseshoes and periodic orbits that we use to approximate a given hyperbolic ergodic measure do belong to the initial homoclinic class.  We do not know whether  given a not necessarily isolated  homoclinic class, every hyperbolic ergodic invariant measure supported on this homoclinic class can always be approximated in such a way.

\section{Proofs}\label{s.proof}

The space $\mathcal{M}$ of all probability Borel measures on $M$ equipped with the weak$^*$-topology is metrizable with the distance $d_{\mathcal{M}}$ given by
\begin{equation}\label{e.dist}
d_{\mathcal{M}}(\mu,\nu)=\sum_{k=1}^{\infty}\frac{1}{2^k}\left|\int \psi_k\,d\mu-\int \psi_k\,d\nu\right|,
\end{equation}
where $\{\psi_k\}_{k\in\mathbb{N}}$ is a dense subset in the unit ball in $C^0(M)$. While the distance defined in this way depends on the choice of the subset $\{\psi_k\}_{k\in \mathbb{N}}$, the topology it generates does not. We will choose the functions $\psi_k$ to be smooth.

\begin{proof}[Proof of Theorem \ref{t.main}]

By a hyperbolic periodic measure $\mu_q$ we mean an atomic ergodic measure equidistributed on a hyperbolic periodic orbit of $q$.

\begin{lemma}\label{l.0}
Let  $q_1, q_2\in\mathcal{H}(p)$ be hyperbolic periodic points with index $s(q_1)=s(q_2)=s(p)$. Then the hyperbolic periodic measures $\mu_{q_1}$ and $\mu_{q_2}$ can be connected in $\mathcal{M}_p^e$ by a continuous path.
\end{lemma}

\begin{proof}[Proof of Lemma \ref{l.0}]
By Condition (H1), the points $q_1$ and $q_2$ are homoclinically related. By the Smale-Birkhoff theorem, there is a hyperbolic horseshoe $\Lambda$\footnote{By a hyperbolic horseshoe we mean a locally maximal hyperbolic set $\Lambda$ which is totally disconnected and such that $f|\Lambda$ is topologically transitive.} that contains both $q_1$ and $q_2$. Lemma \ref{l.0} now follows from the results by Sigmund, see the proof of Theorem B in \cite{S77}.
\end{proof}

We wish to approximate a given hyperbolic measure by periodic measures. There are several results in this direction, see, for example, \cite[Theorem 15.4.7]{BP}. However, we need some specific properties of such approximations that are stated in the following lemma.

\begin{lemma}\label{l.1}
For any $\mu\in\mathcal{M}_p^e$ and any $\varepsilon>0$ the following statements hold:
\begin{enumerate}
\item There exists a hyperbolic periodic point $q\in \mathcal{H}(p)$ with $s(q)=s(p)$ such that we have
$d_{\mathcal{M}}(\mu_q, \mu)<\varepsilon$ for the corresponding hyperbolic periodic measure $\mu_q$;
\item There exists $\delta>0$ such that for any hyperbolic periodic points $q_1, q_2\in \mathcal{H}(p)$ with $s(q_1)=s(q_2)=s(p)$ if the corresponding hyperbolic periodic measures $\mu_{q_1}$ and $\mu_{q_2}$ lie in the $\delta$-neighborhood of the measure $\mu$, then there exists a continuous path
$\{\nu_t\}_{t\in [0,1]}\subset\mathcal{M}_p^e$ with $\nu_0=\mu_{q_1}$,
$\nu_1=\mu_{q_2}$ and such that
$d_{\mathcal{M}}(\nu_t, \mu)<\varepsilon$ for all $t\in [0,1]$.
\end{enumerate}
\end{lemma}

\begin{proof}[Proof of Lemma \ref{l.1}]
Let $\mathcal{R}$ be the set of all Lyapunov-Perron regular points, and for each $\ell\ge 1$ let $\mathcal{R}_\ell$ be the regular set (see \cite{BP} for definitions). There exists $\ell\in\mathbb{N}$ such that 
$\mu(\mathcal{R}_\ell)>0$. Fix $\varepsilon>0$. By Birkhoff's Ergodic Theorem, for a $\mu$-generic point $x\in\mathcal{R}_\ell$ there exists $N\in\mathbb{N}$ such that for any $n>N$
\begin{equation}\label{e.defdist}
d_{\mathcal{M}}\left(\frac{1}{n}\sum_{k=0}^{n-1}\delta_{f^k(x)}, \mu\right)<\frac{\varepsilon}{2}.
\end{equation}
Choose $L\in\mathbb{N}$ such that
$\sum_{k=L+1}^\infty\frac{1}{2^{k-1}}<\frac{\varepsilon}{4}$. Let $\{\psi_k\}$ be the dense collection of smooth functions $\{\psi_k\}$ in the definition \eqref{e.dist} of the distance $d_{\mathcal{M}}$ and let $C=C(\varepsilon)$ be the common Lipschitz constant of the functions $\{\psi_1,\dots, \psi_L\}$. Let $U(\mathcal{H}(p))$ be a neighborhood of $\mathcal{H}(p)$ such that 
$\mathcal{H}(p)=\bigcap_{n\in \mathbb{Z}}f^n(U(\mathcal{H}(p)))$; its existence is guaranteed by (H2). Let us now choose $\delta>0$ such that $C\delta<\frac{\varepsilon}{4}$ and $\delta$-neighborhood of $\mathcal{H}(p)$ is in $U$. Since $\mu$ is a hyperbolic measure, by \cite[Theorem 15.1.2]{BP}, there exists $n>N$ and a hyperbolic periodic point $y\in M$ of period $n$ such that $\text{dist}_M(f^k(x), f^k(y))<\delta$ for all $k=0, \dots, n-1$ and 
$s(y)=s(x)$. Then the orbit of $y$ is in $U$, and hence belongs to 
$\mathcal{H}(p)$. Also, we have
\begin{equation}\label{e.sums}
\begin{aligned}
d_{\mathcal{M}}&\left(\frac{1}{n}\sum_{k=0}^{n-1}\delta_{f^k(x)}, \frac{1}{n}\sum_{k=0}^{n-1}\delta_{f^k(y)}\right)\\
 \le&\sum_{k=1}^{L}\frac{1}{2^k}\left|\int \psi_kd\left(\frac{1}{n}\sum_{i=0}^{n-1}\delta_{f^i(x)}\right)-\int \psi_kd\left(\frac{1}{n}\sum_{i=0}^{n-1}\delta_{f^i(y)}\right)\right|+\sum_{k=L+1}^\infty\frac{1}{2^{k-1}} \\
\le &C\delta+\frac{\varepsilon}{4}<\frac{\varepsilon}{2}.
\end{aligned}
\end{equation}
The first statement of the lemma now follows from \eqref{e.defdist} and \eqref{e.sums}.

To prove the second statement let $q_1, q_2\in \mathcal{H}(p)$ be any hyperbolic periodic points such that $s(q_1)=s(q_2)=s(p)$ and the corresponding hyperbolic periodic measures $\mu_{q_1}$ and $\mu_{q_2}$ lie in the $\delta$-neighborhood of the measure $\mu$. By Conditions (H1) and (H2), the points $q_1$ and $q_2$ are homoclinically related and hence, there is a hyperbolic horseshoe which contains both points. The desired result now follows from \cite{S77} (see the proof of Theorem B).
\end{proof}
We now compete the proof of Theorem \ref{t.main}. Let $\eta$ and
$\widetilde{\eta}\in\mathcal{M}_p^e$ be two hyperbolic ergodic measures. By Statement 1 of Lemma \ref{l.1}, there are sequences of hyperbolic periodic points $q_k$ and $\widetilde{q}_k$ in $\mathcal{H}(p)$ with $s(q_k)=s(\widetilde{q}_k)=s(p)$ such that for the corresponding sequences of hyperbolic periodic measures $\{\mu_{q_k}\}_{k\in\mathbb{N}}$ and
$\{\mu_{\widetilde{q}_k}\}_{k\in\mathbb{N}}$ we have $\mu_{q_k}\to\eta$ and
$\mu_{\widetilde{q}_k}\to\widetilde{\eta}$.
By Lemma \ref{l.0}, there is a path 
$\{\nu_t\}_{t\in\left[\frac{1}{3},\frac{2}{3}\right]}$ in $\mathcal{M}_p^e$ that connects $\mu_{q_1}$ and $\mu_{\widetilde{q}_1}$, that is, 
$\nu_{\frac{1}{3}}=\mu_{q_1}$ and $\nu_{\frac{2}{3}}=\mu_{\widetilde{q}_1}$. By Statement 2 of Lemma \ref{l.1}, for any $k\in\mathbb{N}$ there are paths
$$
\{\nu_t\}_{t\in \left[\frac{1}{3^{k+1}},\frac{1}{3^k}\right]} \text{ and }
\{\nu_t\}_{t\in\left[1-\frac{1}{3^{k}},1-\frac{1}{3^{k+1}}\right]}
$$
in $\mathcal{M}_p^e$ that connect measures $\mu_{q_{k}}$, $\mu_{q_{k+1}}$ and measures $\mu_{\widetilde{q}_{k}}$, $\mu_{\widetilde{q}_{k+1}}$, respectively and the length of each such path does not exceed $\frac{\varepsilon}{2^k}$. Applying again Lemma \ref{l.1}, we conclude that the path $\{\nu_t\}_{t\in [0,1]}$ given by the above choices and such that 
$\nu_0=\eta$ and $\nu_1=\widetilde{\eta}$ is continuous. The desired result now follows.
\end{proof}

\begin{proof}[Proof of Theorem \ref{t.2}]
Arguments similar to those used in the proof of Lemma \ref{l.1} (see also the proof of Theorem B in \cite{S77}) show that the following statement holds:
\begin{lemma}\label{l.new}
For any $\varepsilon>0$ there exists $\delta>0$ such that for any two measures $\mu_1, \mu_2\in \mathcal{M}_p^e$ with $d_{\mathcal{M}}(\mu_1, \mu_2)<\delta$ there exists a continuous path in $\mathcal{M}_p^e$ connecting $\mu_1$ and $\mu_2$ of diameter smaller than $\varepsilon$.
\end{lemma}
Now Theorem \ref{t.2} can be obtained using Lemma \ref{l.new} in the same way Theorem \ref{t.main} was obtained using Lemma \ref{l.1}.
\end{proof}

\begin{proof}[Proof of Theorem \ref{t.entropydense}]
Given a (not necessarily ergodic) measure $\mu\in \mathcal{M}_p$, by the ergodic decomposition, there exists a measure $\nu$ on the space $\mathcal{M}_p^e$ such that
$$
\mu=\int \tau\,d\nu(\tau) \ \ \text{and} \ \ h_\mu=\int h_{\tau} d\nu(\tau).
$$
It follows that for any $\varepsilon>0$ there are measures
$\tau_1,\ldots, \tau_N\in\mathcal{M}_p^e$ and positive coefficients $\alpha_1,\ldots,\alpha_N$ such that
\begin{equation}\label{f.1}
d_{\mathcal{M}_p}\left(\mu,\sum_{k=1}^N\alpha_k \tau_k\right)<\varepsilon \ \ \text{and}\ \ \left|h_\mu - \sum_{k=1}^N\alpha_kh_{\tau_k}\right|<\varepsilon.
\end{equation}
Given a hyperbolic ergodic measure $\tau$ with $h_{\tau}>0$, there exist a sequence of hyperbolic horseshoes $\Lambda_n$ and a sequence of ergodic measures $\{\nu_n\}_{n\in\mathbb{N}}$ supported on $\Lambda_n$ such that
$\nu_n\to \tau$ and $h_{\nu_n}\to h_{\tau}$  as $n\to\infty$, see, for example,  Corollary 15.6.2 in \cite{BP}. By (H2), one can ensure in the construction of these horseshoes that $\Lambda_n\subseteq\mathcal{H}(p)$ and that $\nu_n$ are the measures of maximal entropy and hence, Markov measures. Further, for every $x\in\Lambda_n$ the dimension of the stable manifold through $x$ is equal to the index of $p$.

In the case when $h_{\tau}=0$ the measure $\tau$ can be approximated by a hyperbolic periodic measure supported on an orbit of a hyperbolic periodic point $q\in \mathcal{H}(p)$ (see Lemma~\ref{l.1}) with $s(q)=s(p)$. There exists a hyperbolic horseshoe $\Lambda_q\subset \mathcal{H}(p)$ that contains a periodic point $q$, and one can choose a Markov measure (which is not a measure of maximal entropy in this case) supported on $\Lambda_q$ that is arbitrary close to the atomic invariant measure supported on the orbit of $q$, and has arbirary small entropy (notice that the support of this Markov measure does not have to be close to the orbit of $q$).

It follows from what was said above that to each hyperbolic ergodic measure 
$\tau_k$ we can associate a hyperbolic horseshoe
$\Lambda_k\subseteq \mathcal{H}(p)$ and a Markov measure $\nu_k$ supported on $\Lambda_k$ such that for every $k=1, \ldots, N$ we have
\begin{equation}\label{f.2}
d_{\mathcal{M}_p}\left(\tau_k, \nu_k\right)<\frac{\varepsilon}{N} \ \ \text{and}\ \ \left|h_{\tau_k} - h_{\nu_k}\right|<\frac{\varepsilon}{N}.
\end{equation}
Notice that all horseshoes $\Lambda_k$ have the same index $s(p)$ and that they are homoclinically related (this means that every periodic orbit in one of the horseshoes is homoclinically related to any periodic orbit on the other horseshoe). This implies that there exists a hyperbolic horseshoe 
$\Lambda\subset\mathcal{H}(p)$ that contains all $\Lambda_k$.

The Markov measure $\nu_k$ is constructed with respect to a Markov partition of $\Lambda_k$ that we denote by $\xi_k$. There exists a Markov partition
$\xi$ of $\Lambda$ such that  its restriction on each $\Lambda_k$ is a refinement of $\xi_k$. The measure $\sum_{k=1}^N\alpha_k \nu_k$ is a Markov measure on $\Lambda$ with respect to the partition $\xi$. Notice that Markov measures as well as their entropies depend continuously on their stochastic matrices. Therefore, given an arbitrarily (not necessarily ergodic) Markov measure, one can produce its small perturbation which is an ergodic Markov measure whose entropy is close to the entropy of the unperturbed one. This gives the required approximation of the measure
$\sum_{k=1}^N\alpha_k \nu_k$, which by \eqref{f.1} and \eqref{f.2} is close to the initial measure $\mu$ and whose entropy is close to $h_\mu$.
\end{proof}

\end{document}